\documentclass[12pt,a4paper,twoside]{amsart}
\usepackage{a4wide, charter, amsmath,amsthm, amsbsy,amsfonts,amssymb, bm, manfnt}
\usepackage[mathscr]{eucal}
\usepackage{stmaryrd,mathrsfs,graphicx, amscd, tikz-cd, cancel}

\usepackage[
	linkcolor=blue,
	citecolor=red,
	urlcolor=red,
]{hyperref}

\usepackage[active]{srcltx}
\allowdisplaybreaks
\numberwithin{equation}{section}

\def\CC{{\mathbb C}}

\def\QQ{{\mathbb Q}}

\def\ZZ{{\mathbb Z}}

\def\Ical{{\mathcal I}}

\def\Mcal{{\mathcal M}}

\def\Scal{{\mathfrak S}}

\def\splie{{\mathfrak{sp}}}

\def\ssm{\smallsetminus}

\def\pt{{\scriptscriptstyle\bullet}}

\newcommand\diff{\operatorname{Diff}}

\newcommand\Gr{\operatorname{gr}}

\newcommand\Hom{\operatorname{Hom}}

\newcommand\Ind{\operatorname{Ind}}

\newcommand\Mod{\operatorname{Mod}}

\newcommand\pic{\operatorname{Pic}}

\newcommand\sym{\operatorname{Sym}}

\newcommand\sign{\operatorname{sign}}

\newcommand\Tor{\operatorname{Tor}}

\newtheorem{theorem}{Theorem}[section]
\newtheorem{lemma}[theorem]{Lemma}
\newtheorem{proposition}[theorem]{Proposition}
\newtheorem{corollary}[theorem]{Corollary}

\numberwithin{definition}{section}

\theoremstyle{remark}

\newtheorem{remark}[theorem]{Remark}

\title[Torelli action on the configuration space of a surface]{Torelli group action on the configuration space of a surface}
\author{Eduard Looijenga}
\address{Yau Mathematical Sciences Center, Tsinghua University, Beijing; Mathematisch Instituut, Universiteit Utrecht, Nederland}
\thanks{Supported by the Chinese National Science Foundation.}
\begin{document}
\begin{abstract}
We show that the Torelli group of a closed surface of genus $\ge 3$ acts nontrivially on the rational  cohomology of its space of $3$-element subsets.
This implies that if the surface has the structure of a Riemann surface, the mixed Hodge structure on the space of its positive reduced divisors of degree 
$3$ does in general not split over $\QQ$.
\end{abstract}
\maketitle

\section{Introduction} 

Given a manifold $X$ and an integer $n>0$, then its $n$-point configuration space $F_n(X)$ is by definition  the open subset in $X^n$ with all its components distinct.
It is clear that this open subset is preserved by the diffeomorphism group $\diff(X)$ of $X$  and that the resulting action on $H^\pt(F_n(X))$ factors through the mapping class group
$\Mod(X):=\pi_0(\diff(X))$ of $X$ (\footnote{Unless otherwise specified, in  this note  (singular) homology and cohomology is taken with $\QQ$-coefficients.}). 

When $X$ is a complex projective manifold, there is a spectral sequence \cite{totaro} which produces the weight filtration on $H^\pt(F_n(X))$. It  shows that $\Gr^W_\pt H^\pt(F_n(X))$ only depends on $H^\pt(X)$, that the action of $\Mod(X)$ on $H^\pt(F_n(X))$ preserves the weight filtration and that its action on $\Gr^W_\pt H^\pt(F_n(X))$ is through its representation on $H^\pt(X)$.
It has been claimed (\cite{gorinov}, now withdrawn)  that for such an $X$, the weight filtration
on $H^\pt(F_n(X))$ is split. As such a splitting is necessarily unique, this would imply the much stronger assertion that what is true for $\Gr^W_\pt H^\pt(F_n(X))$ is in fact true for $H^\pt(F_n(X))$:  $H^\pt(F_n(X))$ would then only depend on 
$H^\pt(X)$ and $\Mod(X)$ would act on $H^\pt(F_n(X))$ through its representation on $H^\pt(X)$.
We prove that this is not the case. In fact, we show that when $X$ is a closed  orientable surface of genus $\ge 3$, then its Torelli group acts nontrivially on 
$H^3(F_3(X))^{\Scal_3}\cong H^3(F_3(X)/\Scal_3)$, where $F_3(X)/\Scal_3$ has of course the interpretation as the space of $3$-element subsets of $X$.
 As this is a purely topological assertion,  we kept the discussion of our example in that spirit, thereby excising all Hodge theory, but in the last section we describe some of the implications for the mixed Hodge structure on $F_3(X)/\Scal_3$ when $X$ comes with a complex structure, making it a compact Riemann surface.  
 
Our interest in this issue actually originates here, for in the case of  Riemann surfaces we envisage an application to WZW theory: here 
the degree $n$-part of the Hodge filtration on  $H^n(F^n(X\ssm P))$ (with $P$ is a possibly empty finite subset of $X$) is relevant.

Some previous work in this area connected with this paper was concerned with the case that the surface $X$ has  a boundary component $\partial X$: 
Moriyama \cite{moriyama} proves among other things that the Torelli group then acts on its rational cohomology of $F_n(X)$ `as unipotently as possible' and Bianchi \cite{bianchi} recently showed that if $X$ has genus $\ge 2$, then for $n\ge 2$, the representation of $\Mod(X, \partial X)$ on $H_2(F_n(X))$ does not factor through its symplectic representation on $H_1(X)$.
\\

\emph{Acknowledgments.} I thank Dick Hain, Dan Petersen and  Alexey Gorinov for correspondence on this issue. I thank  Dick and  Benson Farb  for drawing my attention to \cite{gorinov} resp.\ \cite{moriyama}.

Dick  has informed me that he also verified  that the Torelli group of a surface can act nontrivially on the cohomology of its configuration space. After posting the first version, I learned from Oscar Randal-Williams that similar results have been obtained by his student Andreas Stavrou and  will appear in his forthcoming doctoral Cambridge thesis. Oscar also pointed me to the work of Bianchi.

\section{The configuration space of 3 points on a surface}
In this section, $X$ is a closed connected (differentiable) surface of genus $\ge 2$; we later assume $g\ge 3$. 
We write $V$ of the symplectic vector space $H_1(X)$ and let $\delta\in \wedge^2V$ represent the intersection pairing on $V^\vee=H^1(X)$.
We have a $\splie(V)$-equivariant splitting  $\wedge^3V=\wedge^3_oV\oplus \delta\wedge V$, where $\wedge^3_oV$ if the kernel of the map 
$\wedge^3V\to V$ defined by $a_0\wedge a_1\wedge a_2\mapsto \sum_{i\in\ZZ/3} (a_i\cdot a_{i+1})a_{i+2}$. The summand $\wedge^3_oV$
is additively spanned by the $a_0\wedge a_1\wedge a_2$ for which the intersection products $a_i\cdot a_j$ are all zero.
For $g\ge 3$, it is irreducible  as a  $\splie(V)$-representation (it represents the third fundamental representation).

The following lemma (the proof of which is left to the reader)  shows among other things that the rational homology of $F_2(X)$ in degree $2$ is simply expressed in terms of the rational cohomology of $X$.

\begin{lemma}\label{lemma:emb}
The inclusion $F_2(X)\subset X^2$ identifies  $H_2(F_2(X))$  with the classes in $H_2(X^2)$ that have zero intersection number with the diagonal.
In particular, it induces an isomorphism between the spaces of $\Scal_2$-anti-invariants $H_2(F_2(X))^{\sign}\cong H_2(X^2)^{\sign}$, where we note that the latter space  is via the K\"unneth decomposition identified with the direct
sum of $\sym^2V$ and the span of $[X]\times 1-1\times [X]$.  $\square$
\end{lemma}


Let  $f_{12}: X^2\to X^3$ be the map $f_{12}(z,z')=(z,z,z')$. 
If $D_{123}\subset X^3$ stands for the main diagonal, then $f_{12}$ maps $F_2(X)$ isomorphically onto $D_{12}\ssm D_{123}$.  Choose a tubular neighborhood boundary $E\subset F_3(X)$ over $D_{12}\ssm D_{123}$ and  regard this via $f_{12}$ as an oriented circle bundle over $F_2(X)$. This gives rise to a \emph{Lefschetz tube mapping} 
\[
T: H_2(F_2(X))\to H_3(F_3(X))
\]
which assigns to a $2$-cycle on $F_2(X)$ its preimage in $E$, viewed as a $3$-cycle on $F_3(X)$. This map is not equivariant with respect to the $\Scal_2$-action, but  takes $\Scal_2$-invariants to $\Scal_2$-anti-invariants and vice versa.
We extract from this a $\Scal_3$-equivariant map, by bringing the $\Scal_3$-translates of $D_{12}$ (whose union is the `fat' diagonal of $X^3$) into play.:
\[
\Ind^{\Scal_3}_{\Scal_2} (H_2(F_2(X))\otimes\sign)\xrightarrow{T} H_3(F_3(X))
\]
By passing to $\Scal_3$-invariants, we then get a map $H_2(F_2(X))^{\sign}\to H_3(F_3(X))^{\Scal_3}$.
 In view of Lemma \ref{lemma:emb}, we may identify $H_2(F_2(X))^{\sign}$ with $H_2(X^2)^{\sign}$.
This results in a map $H_2(X)^{\sign}\to H_3(F^3(X))^{\Scal_3}$, which we still denote by $T$. It  appears in the proposition below. 

\begin{proposition}\label{prop:invariants}
We  have an exact sequence 
\[
0\to H_2(X)^{\sign}\xrightarrow{T}  H_3(F^3(X))^{\Scal_3}\to \wedge_o^3 V\to 0,
\]
and the $\Scal_3$-invariant supplement of $H_3(F^3(X))^{\Scal_3}$ in $H_3(F^3(X))$ embeds in $H_3(X^3)$.
\end{proposition}

We omit the proof, as this follows in a rather straightforward manner from the Leray spectral sequence for the embedding $F_3(X)\subset X^3$ and then dualizing
(see for instance \cite{looij}, \cite{totaro})  or by what is essentially equivalent to this,  a study of the long exact sequence for homology of the triple $(X^3, X^3\ssm D_{123}, F^3(X))$ and the Gysin sequences with $\Scal_3$-action it gives rise to.

Our example is based on the following corollary.

\begin{corollary}\label{cor:T}
If $a, b\in V$ have zero intersection number, then $a\times b$ represents an element of $H_2(F_2(X))$ and if both
are nonzero, then $T(a\times b)\not= 0$. 
\end{corollary}
\begin{proof}
The first assertion follows from Lemma \ref{lemma:emb}. If $a$ and $b$ are both nonzero, then by Proposition \ref{prop:invariants}
the $\Scal_3$-invariant component of $T(a\times  b)$ is $\tfrac{1}{2}T(a\times b+b\times a)\not=0$ ,  so that $T(a\times b)\not=0$.
\end{proof}

We now assume  that  $a$ and $b$ are given by embeddings $S^1\hookrightarrow X$ whose images are disjoint nonseparating curves.
 
\begin{lemma}\label{lemma:homology}
Let $a_+, a_-: S^1\to X$ be embeddings  $C^\infty$-close to, but disjoint with $a$,  with $a_+$ to the right and  $a_-$ to the left of $a$.
Then $T(a\times b)$ is represented by $a\times a_+\times b-a\times a_-\times b$. 
\end{lemma}
\begin{proof}
We think of $S^1$ as the unit circle in $\CC$ and extend $a$ to an embedding $\tilde a$ from the closed 
 annulus $Z\subset \CC$ defined by $\tfrac{1}{2}\le |z|\le 2$ in $X\ssm b$ and assume $a_\pm$ given  by  $a_\pm (z)=\tilde a(2^{\pm1}z)$.
Note that $T(a\times b)$ has the form $T_{X}(a)\times b$, where now $T_X$ is the Lefschetz tube map for the diagonal embedding $X\hookrightarrow X^2$.
So it suffices to prove that $T_{X}(a)$ is represented by $a\times a_+-a\times a_-$. As this is an issue that can be dealt with on $Z$, we may replace $X$ by $Z$. Consider the map $F: S^1\times (Z\ssm\{1\}) \to F_2(Z)$, $F(z,z'):=(z, zz')$. 
If $S_{r}(z)\subset \CC$ stands for the circle of radius $r$ centered at $z\in \CC$, then $T_Z(a)$ is represented by $F|S^1\times S_{1/4}(1)$. It is clear that in $Z\ssm\{1\}$, 
$S_{1/4}(1)$ is homologous to $S_{2}(0)-S_{1/2}(0)$. The map $F$ takes $S^1\times S_{2^{\pm 1}}(0)$ to itself  and hence $T_Z(a)$ is represented by 
$S^1\times (S_{2}(0)-S_{1/2}(0))$.
\end{proof}

\subsection*{The example} Let $d,d'$ be a so-called bounding pair on $X$: both are (a priori nonoriented) nonseparating mutually disjoint curves whose union  splits the surface into two pieces. The Dehn twists $\tau_d$ and  $\tau_{d'}$ defined by $d$ and $d'$, induce the same action on $H_1(X)$ and so the action of $\tau_d\tau_{d'}^{-1}$ on $H_1(X)$ is trivial: it is an element of the Torelli group. Note that $\tau_d\tau_{d'}^{-1}$ is represented by a diffeomorphism  with support in a tubular neighborhood $U$ of $d\cup d'$. 
We now assume that $g\ge 3$ and  that $d\cup d'$ bounds a genus 1 surface $X'\subset X$. We show that  $\tau_d\tau_{d'}^{-1}$ acts nontrivially on $H_3(F_3(X^3))$.

We give both $d$ and $d'$ the induced orientation (so that $d+d'$ is null homologous in $X$). The assumption that $g\ge 3$ allows us find 
oriented nonseparating curves $a,b,c$  on $X$ that are pairwise disjoint, with $a$ contained in $X'\ssm  U$, $b$ contained in $X\ssm (X'\cup U)$ and 
$c$ meeting $d$ and $d'$ transversally in a single point with intersection numbers $1$ resp.\ $-1$ (see Figure 1). Then $a\times c\times b$ defines a $3$-cycle in $F_3(X)$.
We note that $\tau_d\tau_{d'}^{-1}$ takes the cycle $c$  to a cycle which in $U$ is homologous to $c+d+d'$. Hence 
$\tau_d\tau_{d'}^{-1} (a\times c\times b)$ is in $F_3(X)$ homologous to $a\times c\times b$ + $a\times (d+d')\times b$. It remains to see that
$a\times (d+d')\times b$ is nonzero in $F_3(X)$. For this we note that $X'$ contains a subsurface with boundary $d+d'\pm (a_+-a_-)$.
It follows that $a\times (d+d')\times b$ is in $F_3(X)$ homologous to $\pm a\times (a_+-a_-)\times b$ and hence, by Lemma \ref{lemma:homology}, to $\pm T(a\times b)$.  According to Corollary \ref{cor:T} this is nonzero.
\\
\begin{figure}\label{fig:example}
\centerline{\qquad\qquad\includegraphics[scale=0.15]{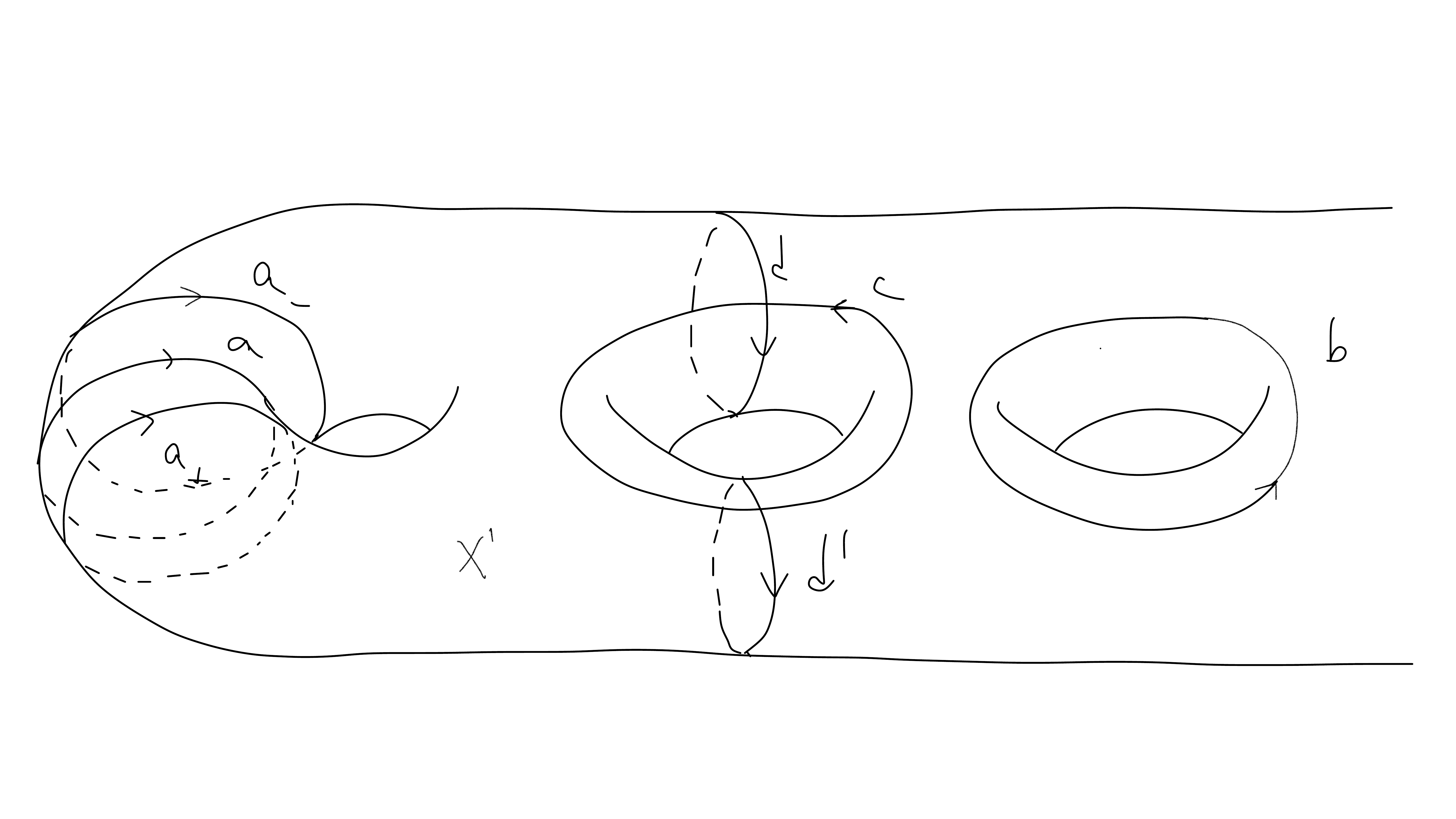}}
\caption{\footnotesize The surface $X$ with the closed curves.}
\end{figure}

\subsection*{The action of the Johnson group} We return to the exact sequence of Proposition \ref{prop:invariants}. Let us identify $\QQ$ with the $\QQ$-span of
$[X]\times 1-1\times [X]$, so that $H^2(X^2)^{\sign}=\QQ\oplus \sym^2V$.
The mapping class  group $\Mod(X)$ acts on this sequence with its action is on the noncentral terms via its symplectic representation 
$\Mod(X)\to\splie (V)$. Hence the Torelli group $\Tor(X)$, which is by definition the kernel of the homomorphism $\Mod(X)\to \splie(V)$, will act unipotently on the middle term. Such an action is given by a group homomorphism 
\[
\Tor(X)\to \Hom (\wedge^3_oV, \QQ\oplus \sym^2V)\cong\Hom (\wedge^3_oV, \QQ)\oplus \Hom (\wedge^3_o V, \sym^2 V).
\]
This action will of course be via the
abelianization of $\Tor(X)$ tensored with $\QQ$. Dennis Johnson \cite{johnson} has identified that group in a natural, $\splie (H_1(X; \ZZ))$-equivariant 
manner, with the 
vector space  $\wedge^3_o V$; we will therefore refer to $\wedge^3_o V$ as the \emph{Johnson group} of $X$. So this gives  maps of $\splie (H_1(X; \ZZ))$representations
\[
\wedge^3_oV\otimes \wedge^3_oV \to \QQ, \quad \wedge^3_oV\otimes \wedge^3_oV \to \sym^2 V.
\]
Since $\splie (H_1(X; \ZZ))$ is generated by its infinite cyclic subgroups of additive type, any  tensor that is $\splie (H_1(X; \ZZ))$-invariant is automatically $\splie(V)$-invariant. So these must be maps of $\splie(V)$ representations. Up to scalar there is in each case only one such map: 
in the first case the invariant symplectic form on $\wedge^3_oV$ defined by 
\[
(a_0\wedge a_1\wedge a_2)\otimes (b_0\wedge b_1\wedge b_2)\mapsto
\sum_{\tau\in \Scal_3}  \sign(\tau)(a_0\cdot b_{\tau(0)})(a_1\cdot b_{\tau(1)})(a_2\cdot b_{\tau(2)})
\] 
and in the second case, the one induced  by 
\[
\Phi: \textstyle (a_0\wedge a_1\wedge a_2)\otimes (b_0\wedge b_1\wedge b_2)\mapsto \sum_{i, j\in \ZZ/3}  q(a_i\wedge a_{i+1}, b_j\wedge b_{j+1}) a_{i+2}b_{j+2},
\]
where we used the natural (invariant) symmetric bilinear form on $\wedge^2V$ defined by 
\[
q(a\wedge a', b\wedge b')=(a\cdot b)(a'\cdot b')-(a \cdot b')(a'\cdot b).
\] 
Our example shows that  one of these maps is nonzero. In order to determine which one, we recall how to make the Johnson identification explicit.
If $(d,d')$ is a bounding pair on $X$, denote by $X'$ and $X''$ the two compact subsurfaces  whose union is $X$ and whose intersection is the given bounding pair. Then
$H_1(X')\to H_1(X)=V$ is an embedding  whose image is degenerate with respect to the intersection pairing, and if we give $d$ 
the orientation induced by $X'$, then  its kernel is spanned by $d$.  The intersection pairing on $H_1(X')$ defines an nonzero element $\delta_{X'}\in \wedge^2H_1(X')/\QQ d$. Wedging this with $d$  gives a well-defined element  $j(X'):=d\wedge \delta_{X'}\in \wedge^3 V$. One checks that $j(X')-j(X'')=d\wedge \delta$,  so that the projection of $j(X')$ in $\wedge_o^3V$ only depends on $(d,d')$. The resulting element $j(d,d')\in \wedge_o^3V$ then represents the image of
$\tau_d\tau_{d'}^{-1}$ in $\wedge_o^3V$.

Returning then to our example, let $a'\in V$ be represented by a cycle on $X'$ for which $a\cdot a'=1$. Then  $j(X')=d\wedge a\wedge a'$. This element 
is not in $\wedge^3_o V$, but since $a,b,c$ span an isotropic subspace,  $a\wedge c\wedge b$ is. Hence
\[
\Phi (j(d,d'), a\wedge c\wedge b)=\Phi(d\wedge a\wedge a', a\wedge c\wedge b)=q(a'\wedge d, a\wedge c)ab=(a'\cdot a)(d\cdot c)ab=ab.
\]
Up to sign this is indeed the variation of $a\wedge c\wedge b$ under $\tau_d\tau_{d'}^{-1}$. So if we denote by 
$H^o_3(F_3(X))^{\Scal_3}$ the quotient of $H_3(F_3(X))^{\Scal_3}$ obtained by dividing out by the copy of $\QQ$ in $H_2(X^2)^{\sign}$, then
the action of the Johnson group on it is nontrivial.

\section{The holomorphic setting and its mixed Hodge theory}

The map $T$  being nonzero in Corollary \ref{cor:T} has the dual interpretation of a residue being nonzero. To make this more appealing to algebraic geometers, assume $X$ endowed with a complex structure inducing the given orientation and denote the resulting compact Riemann surface $C$.
We  denote diagonal of $C^2$ by $D_2$ and the fat diagonal of $C^3$ by $D_3$. We define the  sheaf $\Omega_{C^3}(\log D_3)$ of logarithmic $3$-forms of the pair $(C^3, D_3)$ as the subsheaf
of $\Omega_{C^3}(\Delta)$ spanned by the $\Scal_3$-translates of $\Omega_{C^3}(D_{12}+D_{23})$. Then $f_{ij}$ defines a residue map
$\Omega^3_{C^3}(\log D_3)\to f_{ij*}\Omega^2_{C^2}(D_2)$. These  fit in an exact sequence 
\[
0\to \Omega^3_{C^3}\to \Omega^3_{C^3}(\log D_3)\to \oplus_{i<j}f_{ij*}\Omega^2_{C^2}(D_2)\to 0.
\]
A $2$-form on $C^2$ with only a simple pole along the diagonal is regular (for the residue sum relative to a projection on a factor must be zero). Hence the long exact cohomology sequence begins with
\[
\textstyle 0\to H^0(C^3,\Omega_{C^3})\to H^0(C^3, \Omega_{C^3}(\log D_3))\to H^0(C,\Omega^2_{C^2})^3\to H^1(C^3,\Omega^3_{C^3})
\]
The space $H^0(C^3, \Omega_{C^3}(\log D_3))$ has the  mixed Hodge theory interpretation as $F^3H^3(F_3(C))$.
We explicate the $\Scal_3$-action on this sequence by identifying the term $H^0(C,\Omega^2_{C^2})^3$ as the induced representation $\Ind^{\Scal_3}_{\Scal_2}
(H^0(C^2,\Omega^2_{C^2})\otimes \sign)$.  
The K\"unneth formula identifies $H^1(C^3,\Omega^3_{C^3})$ with the direct sum of the three $\Scal_3$-translates of 
$H^0(C^2, \Omega^2_{C^2})\times \mu$, where $\mu\in H^1(C, \Omega_C)$ is the orientation class, so that $H^1(C^3,\Omega^3_{C^3})$ yields another copy of $\Ind^{\Scal_3}_{\Scal_2} H^0(C,\Omega_C^2)$. The last map is then  given by 
\[
(\alpha_1, \alpha_2, \alpha_3)\mapsto (\alpha_1+\sigma^*\alpha_2, \alpha_2+\sigma^*\alpha_3, 
\alpha_3+\sigma^*\alpha_1),
\]
where $\sigma$ is the transposition map. Its kernel consists of the $(\alpha, \alpha, \alpha)$ with $\alpha+\sigma^*\alpha=0$
(which is equivalent to $\alpha\in \sym^2H^0(C,\Omega_C)$).
So in order that a regular $2$-form $\alpha$ on  $C^2$ appears as the residue of a logarithmic $3$-form $\omega$ on $F_3(C)$, it is necessary and sufficient that $\sigma^*\alpha=-\alpha$, and then $\omega$ will have the same residue on the other diagonal divisors in the sense that
the three residues make up an element of $H^0(C,\Omega_C^2)^3$ that is alternating with respect to the $\Scal_3$-action.  To be concrete: given a local coordinate $z$  at some  $p\in C$, then  $\omega$ has at $(p,p,p)$ the form
\[
\big(\frac{1}{z_1-z_2}+\frac{1}{z_2-z_3}+\frac{1}{z_2-z_3}\big)f(z_1,z_2,z_3)dz_1\wedge  dz_2\wedge dz_3
\]
with $f$ symmetric in its arguments when restricted to the fat diagonal. We can then take $\omega$ to be $\Scal_3$-invariant.  Since the space of symmetric regular $3$-forms on $C^3$ is  $\wedge^3H^0(C, \Omega_C)$, we have
an exact sequence
\[
\textstyle 0\to \wedge^3H^0(C, \Omega_C)\to H^0(C^3, \Omega^3_{C^3}(\log \Delta))^{\Scal_3}\to \sym^2H^0(C,\Omega_C)\to 0
\]
and the $\Scal_3$-invariant supplement of  $H^0(C^3, \Omega^3_{C^3}(\log \Delta))^{\Scal_3}$ is contained in $H^0(C^3, \Omega^3_{C^3})$. The Betti version of the above sequence is
\[
\textstyle 0\to \wedge^3_o V^\vee\to H_o^3(F^3(C))^{\Scal_3}\to \sym^2 V^\vee(-1)\to 0,
\]
where the middle term is the dual of the space $H^o_3(F^3(C))^{\Scal_3}$ in the previous section. It appears here as
a mixed Hodge substructure of $H^3(F_3(C))^{\Scal_3}$ whose  complexification contains $F_3H^3(F_3(C))$. The  exact sequence above
defines the weight filtration on $H_o^3(F^3(C))^{\Scal_3}$: $\wedge^3_o V^\vee$ is pure of weight $3$ and the quotient is of weight $4$. 
The second map is a residue map and dual to $T$. Our previous discussion shows that  for a general choice of complex structure this is not the direct sum of two pure Hodge structures. We already determined $F^3H^3_o(F^3(C))^{\Scal_3}$. One may verify that the other members of the Hodge filtration on 
 $F^3H^3_o(F^3(C))^{\Scal_3}$ can inductively be given by the action of the Johnson group: $F^{i-1}$ is the span of $F^i$ and the image of $F^i$ under $\wedge^3_oV$.

\begin{remark}\label{rem:}
Note that  in this setting  $F_3(C)/\Scal_3$ can be thought of as the base of the universal reduced positive divisor of degree $3$. In case $C$ is nonhyperelliptic of genus $3$ (in other words, a smooth quartic plane curve), then the natural map $C^3/\Scal_3\to \pic^3(C)$ is a birational morphism whose exceptional fibers are of dimension one (the exceptional set in $\pic^3(C)$ is the image of 
$z\in C\mapsto 3(z)\in\pic^3(C)$ and given $z\in C$, then the degree $3$ divisors on $C$ defined by the pencil through $z$ make up an exceptional fiber).  The locus of reduced divisors defines a surface $\Theta_{2,1}\subset \pic^3(C)$ (the image of the map
$(z,z')\in C^2\mapsto 2(z)+(z')\in\pic^3(C)$), and it  then follows that $H^3_o(F^3(C))^{\Scal_3}$ also appears in the cohomology of $\pic^3(C)\ssm \Theta_{2,1}$. Since $\pic^3(C)$ is a torsor for the jacobian of $C$ and almost every principally polarized abelian $3$-fold is such a jacobian, this is in fact a property of a general principally polarized abelian $3$-fold.
\end{remark}

\end{document}